\documentclass[12pt]{amsart}

\usepackage[ansinew]{inputenc}

\usepackage{amsmath,amssymb,verbatim,url}
\usepackage{mathrsfs,color}
\usepackage{enumerate,graphicx}
\usepackage{epstopdf}
\usepackage[active]{srcltx}

\newtheorem{theorem}{Theorem}[section]

\newtheorem{lemma}[theorem]{Lemma}
\newtheorem{proposition}[theorem]{Proposition}

\newtheorem{example}[theorem]{Example}

\newcommand{\hyper}[5]{\,{}_{#1}F_{#2}\left(\!\!%
\begin{array}{cc}{\displaystyle{#3}}\\[-0.1ex]
{\displaystyle{#4}} \end{array}\Big| \,{\displaystyle{#5}}
\right)}

\date{\today}

\title[Orthogonal polynomials on quadratic lattices]{Characterizations of classical orthogonal polynomials on quadratic lattices}

\author[Njinkeu]{Marlyse Njinkeu Sandjon}
\address[Njinkeu]{Department of Mathematics, Higher Teachers' Training College, University of Yaounde I, Cameroon, and African Institute for Mathematical Sciences, AIMS-Cameroon, P.O. Box 608, Limb\'e Crystal Gardens, South West Region, Cameroon}
\email[Njinkeu]{msnjinkeu@yahoo.fr}

\author[Branquinho]{Am\'\i lcar Branquinho}
\address[Branquinho]{CMUC, Department of Mathematics, University of Coimbra
Apartado 3008, EC Santa Cruz, 3001--501 Coimbra, Portugal.}
\email[Branquinho]{ajplb@mat.uc.pt}

\author[Foupouagnigni]{Mama Foupouagnigni}
\address[Foupouagnigni]{Department of Mathematics, Higher Teachers' Training College, University of Yaounde I, Cameroon, and African Institute for Mathematical Sciences, AIMS-Cameroon, P.O. Box 608, Limb\'e Crystal Gardens, South West Region, Cameroon}
\email[Foupouagnigni]{mfoupouagnigni@aims-cameroon.org}

\author[Area]{Iv\'an Area}
\address[Area]{Departamento de Matem\'atica Aplicada II,
       E.E. Telecomunicaci\'on,
       Universidade de Vigo,
       Campus Lagoas-Marcosende,
       36310 Vigo, Spain.}
\email[Area]{area@uvigo.es}

\begin{document}

\begin{abstract}
This paper is devoted to characterizations classical orthogonal polynomials on quadratic lattices by using a matrix approach. In this form we recover the Hahn, Geronimus, Tricomi and Bochner type characterizations of classical orthogonal polynomials on quadratic lattices. Moreover a new characterization is also presented. From the Bochner type characterization we derive the three-term recurrence relation coefficients for these polynomials.
\end{abstract}

\maketitle


\section{Introduction}

Classical continuous orthogonal polynomial sequences can be characterized by different properties, using different approaches. Probably the first results in this direction go back to Bochner \cite{MR1545034}, Favard \cite{MR0481884} and Hahn \cite{MR0030647}. Moreover, some recent characterizations can be found in \cite{MR1342384,MR2653053,MR2836867}, by using either differential operators as Bochner or linear functionals as introduced by Maroni \cite{MR932783,MR803215}. Recently a new characterization of classical continuous, discrete and their $q$-analogues was given by Verde-Star \cite{MR3028602,MR3134272} by using a matrix approach.

\medskip
A general presentation of classical continuous orthogonal polynomials in terms of solutions of certain differential equations have been done by Nikiforov et al. \cite{MR1149380,0378.33001,MR922041}. In this direction, classical orthogonal polynomials are solution of the second order linear differential equation
\begin{equation}\label{eq:chyp}
\sigma(x) y^{\prime\prime}(x) + \tau(x) y^{\prime} (x) + \lambda y(x)=0
\end{equation}
where $\sigma$ and $\tau$ are polynomials of at most second and first degree, respectively.

\medskip
The differential equation (\ref{eq:chyp}) can be replaced by a difference equation, giving rise to classical orthogonal polynomials of a discrete variable \cite[Chapter 2]{MR1149380}, if we consider a discretization with constant mesh, or classical orthogonal polynomials on nonuniform lattices \cite[Chapter 3]{MR1149380} if we consider a class of lattices with variable mesh $\mu(t)$. We would like to notice that divided-difference operators associated with the special non-uniform lattices have appeared in many studies of orthogonal polynomials of a discrete variable. For example see the early studies by Hahn  \cite{MR0035344,MR0030647,MR0032858,MR0040557}, the foundational work by Askey and Wilson \cite{MR783216} and the monograph of Nikiforov, Suslov and Uvarov \cite{MR1149380}.

\medskip As indicated in \cite[Theorem 1, page 59]{MR1149380} some restrictions must be imposed on the lattice $\mu(t)$ giving rise to the following classification of the lattices:
\begin{enumerate}
\item Linear lattices if $\mu(t)=c_{2}t+c_{3}$ with $c_{2} \neq 0$.
\item Quadratic lattices if  $\mu(t)=c_{1}t^{2}+c_{2}t+c_{3}$, with $c_{1} \neq 0$.
\item $q$-linear lattices if  $\mu(t)=c_{5}q^{t}+c_{6}$, with $c_{5} \neq 0$.
\item $q$-quadratic lattices if $\mu(t)=c_{4}q^{t}+c_{5}q^{-t}+c_{6}$ with $c_{5}c_{6} \neq 0$.
\end{enumerate}
The characterization theorems of classical orthogonal polynomials in the cases of linear and $q$-linear lattices by using matrix approach have been obtained in \cite{MR3134272}. We would like to emphasize that this approach has not been used in the case of quadratic or $q$-quadratic lattices, despite the importance in many applications of the families belonging to these classes (e.g. Racah or Wilson orthogonal polynomials).

\medskip
In a recent paper \cite{MR2653053} the authors gave a characterization theorem for classical orthogonal polynomials on a lattice as described above by using the Pearson-type equation. Moreover, in \cite{MR2836867} and by using the functional approach, the authors stated and proved a characterization theorem for classical orthogonal polynomials on non-uniform lattices including the Askey-Wilson polynomials.

\medskip
The main aim of this paper is to present a new characterization of classical orthogonal polynomials on quadratic lattices, by using a matrix approach. In doing so, we reinterpret in matrix form previous characterizations classical orthogonal polynomials on quadratic lattices, showing that previous results of \cite{MR3028602,MR3134272} on classical continuous orthogonal polynomials, discrete and their $q$-analogues, can be generalized to nonuniform lattices. In this way, we obtain the Hahn, Geronimus, Tricomi, and Bochner type characterizations. Moreover, by using the method presented by Vicente Gon\c{c}alves, we explicitly obtain  the coefficients in the three-term recurrence relation satisfied by classical orthogonal polynomials on nonuniform lattices from the second order linear divided-difference equation they satisfy.

\medskip
This work is organized as follows: in section \ref{section:2} we introduce the basic definitions and notations. In section \ref{section:3} we reinterpret the Hahn, Geronimus, Tricomi, and Bochner characterizations of classical orthogonal polynomials on quadratic lattices by using a matrix approach and derive a new characterization of these polynomials. Finally, in section \ref{section:4} we extend the method of Vicente Gon\c{c}alves to obtain the coefficients of the three-term recurrence relation of classical orthogonal polynomials on quadratic lattices from the second-order linear divided-difference equation they satisfy.

\section{Basic definitions and notations}\label{section:2}

Let us consider the quadratic lattice
\begin{equation}\label{eq:lattice}
\mu (t) = c_{1} \, t^2 + c_2 \, t + c_3 ,
\end{equation}
where $c_{1}$, $c_{2}$, and $c_{3}$ are constants and in what follows we shall assume that $c_{1} =1$, i.e. a pure quadratic lattice. Notice that the particular case $c_{1}=0$, i.e. linear lattices, have been considered in \cite{MR3134272}, and as mentioned before our intention is to show that that matrix approach can be followed in the case of nonuniform lattices. Let $P_{n}(\mu(t))$ be a monic polynomial of degree $n$ in the lattice $\mu(t)$,
\begin{equation}\label{eq:1}
P_{n}(\mu(t))=P_{n}=p_{n,n}+p_{n-1,n}\vartheta_{1}(t)+p_{n-2,n} \vartheta_{2}(t)+\cdots+p_{1,n} \vartheta_{n-1}(t)+\vartheta_{n}(t),
\end{equation}
where the basis $\big\{ \vartheta_{n}(t) \big\}_{n \geq 0}$ is defined by
\begin{equation*}
\displaystyle \vartheta_{n}(t) = (-4)^{-n} \, (2 t +1/2+ c_2)_n \, (-2 t + 1/2-c_2)_n,
\end{equation*}
and $(A)_{n}=A(A+1)\cdots(A+n-1)$ with $(A)_{0}=1$ denotes the Pochhammer symbol. Let us further define
\begin{eqnarray}\label{eq:polq}
\mathcal P & = & \left[ \begin{matrix} P_0 & P_1 & P_2 & \cdots \end{matrix} \right]^{\mathsf T} = \left[ \begin{matrix} 1 & p_{1,1} + \vartheta_{1}(t) & p_{2,2} + p_{1,2} \, \vartheta_{1}(t) + \vartheta_{2}(t) & \cdots \end{matrix} \right]^{\mathsf T} \\
 & = & A \, \left[ \begin{matrix} 1 & \vartheta_{1}(t) & \vartheta_{2}(t) & \cdots \end{matrix} \right]^{\mathsf T} ,
\end{eqnarray}
where
\begin{equation}\label{eq:matrixA}
A = \left[ \begin{matrix}
1      & 0 & &   & \\
p_{1,1} & 1 & 0 &   & \\
p_{2,2} & p_{1,2} & 1 & \ddots & \\
       & \ddots  & \ddots & \ddots
\end{matrix} \right] .
\end{equation}
The difference operators ${\mathbb{D}}$ and ${\mathbb{S}}$ \cite{MR973434,MR1379135} are defined by
\begin{align}
{\mathbb{D}} f(t)&=\frac{f(t+1/2)-f(t-1/2)}{\mu(t+1/2)-\mu(t-1/2)}, \label{eq:doperator} \\[3mm]
{\mathbb{S}} f(t)&=\frac{f(t+1/2)+f(t-1/2)}{2} \label{eq:dsoperator}.
\end{align}
Notice that the above divided-difference operators transform polynomials of degree $n$ in the lattice $\mu(t)$ defined in \eqref{eq:lattice} into polynomials of respectively degree $n-1$ and $n$ in the same variable $\mu(t)$.
Since
\begin{equation}\label{eq:divdif}
\mathbb D \, \vartheta_{n}(t) = n \, \vartheta_{n-1}(t) ,
\end{equation}
we have
\begin{eqnarray*}
\mathcal P^{\prime} & := & \left[ \begin{matrix} \mathbb D \, P_1 & \frac{1}{2} \mathbb D \, P_2 & \frac{1}{3} \mathbb D \, P_3 & \cdots \end{matrix} \right]^{\mathsf T}
= \left[ \begin{matrix} 1 & p_{1,2} + \vartheta_{1}(t) & \cdots \end{matrix} \right]^{\mathsf T} \\
& \phantom{:} = & \tilde{A} \, \left[ \begin{matrix} 1 & \vartheta_{1}(t) & \vartheta_{2}(t) & \cdots \end{matrix} \right]^{\mathsf T},
\end{eqnarray*}
where
\begin{equation}\label{eq:matrixAtilde}
 \tilde{A} = \tilde{D} \, A \, D,
\end{equation}
with
\begin{equation}\label{eq:matrixD}
{\tilde{D}} = \left[ \begin{matrix}
0  & 1 &   &       & \\
   & 0 & \frac{1}{2} &       & \\
   &   & 0 & \frac{1}{3}      & \\
   &    &   & \ddots & \ddots
\end{matrix} \right] ,
\qquad
D = \left[ \begin{matrix}
0  &   &   &   & \\
 1 & 0  &   &   & \\
   & 2  & 0 &   & \\
   &   & \ddots & \ddots
\end{matrix} \right] .
\end{equation}

\medskip Let us assume that $\{P_{n}\}_{n \geq 0}$ is a sequence of monic orthogonal polynomials on a quadratic lattice $\mu(t)$. Then, the three-term recurrence relation satisfied by $\{P_{n}\}_{n \geq 0}$ reads~as
\begin{equation}\label{eq:ttrrpp}
\mu(t) P_{n}=P_{n+1} + \beta_{n} P_{n} + \gamma_{n} P_{n-1},
\end{equation}
with initial conditions $P_{0}=1$, $P_{1}=\mu(t)-\beta_{0}$.

\begin{lemma} The three-term recurrence relation \eqref{eq:ttrrpp} can be written in matrix form as
\begin{equation}\label{eq:auxil1}
 L \, A = A \, {\pmb{X^1}},
\end{equation}
where
\begin{equation}\label{eq:x1}
{\pmb{X^1}}:=X + \operatorname{diag} \, \{ f_0 , f_1 , \ldots \} ,
\end{equation}
with
\begin{equation}\label{eq:matrixL}
L = \left[ \begin{matrix}
\beta_0   & 1        &       & \\
\gamma_1 & \beta_1   & 1      & \\
         & \gamma_2 & \beta_2 & \ddots \\
         &          & \ddots  & \ddots
\end{matrix} \right] , \qquad
 X = \left[ \begin{matrix}
0  & 1 &   &       & \\
   & 0 & 1 &       & \\
   &   & 0 & 1      & \\
   &    &   & \ddots & \ddots
\end{matrix} \right],
\end{equation}
and the coefficients $f_{n}$ are defined as
\begin{equation}\label{eq:18}
\mu(t) \, \vartheta_{n}(t) = \vartheta_{n+1}(t) + f_{n} \, \vartheta_{n}(t), \quad n= 0, 1, \ldots,
\end{equation}
and given explicitly by
\begin{equation}\label{eq:fnfinal}
f_{n}=-\frac{c_{2}^2}{4}+\frac{1}{16}  (2 n+1)^2+c_{3}.
\end{equation}
\end{lemma}
\begin{proof}
Using \eqref{eq:18} and the linear independence of $\{ \vartheta_{n}(t) \}_{n \geq 0}$ we get
\begin{align*}
 & \mu(t) \, \left[ \begin{matrix} 1 & \vartheta_{1}(t) & \vartheta_{2}(t) & \cdots \end{matrix} \right]^{\mathsf T}
  =  \big( \operatorname{diag} \, \{ f_0 , f_1 , \ldots \} + X \big) \,
\left[ \begin{matrix} 1 & \vartheta_{1}(t) & \vartheta_{2}(t) & \cdots \end{matrix} \right]^{\mathsf T} \\
 & \phantom{olaolaolaola} = {\pmb{X^1}} \left[ \begin{matrix} 1 & \vartheta_{1}(t) & \vartheta_{2}(t) & \cdots \end{matrix} \right]^{\mathsf T} \,;
\end{align*}
 and from
\begin{equation*}
L \, \mathcal P = \mu(t) \, \mathcal P ,
\end{equation*}
the result follows.
\end{proof}

\begin{lemma}
The following matrix relation holds true
\begin{equation*}
AD\tilde{A}^{-1} = AD \tilde{D}A^{-1}D = A J A^{-1}D = D, \label{eq:doc48)}
\end{equation*}
where
\begin{equation}\label{eq:defJJ}
J =\left[
\begin{matrix}
 0& 0 & 0 & 0 &   \\
  0 & 1 & 0 & 0 &   \ddots \\
  0 & 0 & 1 & 0 &  \ddots \\
   & \ddots & \ddots  & \ddots & \ddots \\
\end{matrix} \right].
\end{equation}
\end{lemma}
\begin{proof}
The result follows from the definitions of the matrices $A$, $D$, $\tilde{A}$, $\tilde{D}$ and $J$.
\end{proof}

\section{Characterizations of classical orthogonal polynomials on quadratic lattices}\label{section:3}

\subsection{Hahn's characterization}
Let us assume that the sequence $\{ \frac{1}{n} {\mathbb{D}} P_n = P_{n}^{\prime} \}_{n \geq 1}$ is also orthogonal (Hahn's characterization). Then, the three-term recurrence relation satisfied by $\{P_{n}^{\prime}\}_{n \geq 1}$
\begin{equation*}
P_{1}^{\prime}=1, \quad P_{2}^{\prime}=x-\beta_{0}^{\prime}, \quad
\mu(t) P_{n}^{\prime}=P_{n+1}^{\prime} + \beta_{n}^{\prime} P_{n}^{\prime} + \gamma_{n}^{\prime} P_{n-1}^{\prime} , \ n \in \mathbb N ,
\end{equation*}
can be written in matrix form as
\begin{equation}\label{eq:ttrrdp}
M \, \mathcal P^{\prime} = \mu \, \mathcal P^{\prime} ,
\end{equation}
with
\begin{equation}\label{eq:matrixM}
M = \left[ \begin{matrix}
\beta_0^{\prime}   & 1        &       & \\
\gamma_1^{\prime} & \beta_1^{\prime}   & 1      & \\
         & \gamma_2^{\prime} & \beta_2^{\prime} & \ddots \\
         &          & \ddots  & \ddots
\end{matrix} \right] .
\end{equation}
Thus, by using the definition of $\mathcal P^{\prime}$ it yields
\begin{equation}\label{eq:chahn}
M \, \tilde{A} 
 = \tilde A \, {\pmb{X^1}}.
\end{equation}
where ${\pmb{X^1}}$ has been defined in \eqref{eq:x1}.

\medskip As conclusion, we have the Hahn-type characterization of classical orthogonal polynomials on quadratic lattices by using the matrix approach
\begin{proposition}
The sequence $\{P_{n}\}_{n \geq 0}$ is 
of classical orthogonal polynomials on the quadratic lattice $\mu(t)$ defined in \eqref{eq:lattice} if and only if \eqref{eq:chahn} holds true, where the matrices $M$, $\tilde{A}$ and ${\pmb{X^1}}$ are defined in \eqref{eq:matrixM}, \eqref{eq:matrixAtilde}, and \eqref{eq:x1} respectively.
\end{proposition}

\subsection{Geronimus' characterization} Classical orthogonal polynomials on quadratic lattices can be also characterized from the following algebraic relation (Geronimus' characterization)
\begin{equation*}
\displaystyle \mathbb S \, P_n = P_{n+1}^{\prime} + \ell^1_n \,  P_{n}^{\prime}  + \ell^2_n \, P_{n-1}^{\prime},
\end{equation*}
i.e. each element of the sequence $\{ \mathbb S P_n \}_{n \geq 0}$ can be expressed as a linear combination of three consecutive elements of the sequence $\{ P_n^{\prime} \}_{n \geq 1}$. By using that
\begin{equation}\label{eq:gnfinal}
\mathbb S \, \vartheta_n(t) = \vartheta_{n}(t) + g_n \vartheta_{n-1}(t), \qquad g_n = \frac{n \, (2 \, n -1)}{4},
\end{equation}
the Geronimus characterization can be written in matrix form as
\begin{equation} \label{eq:ger}
\mathbb S \, \mathcal P = U \, \mathcal P^{\prime},
\end{equation}
where
\begin{equation}\label{eq:matrixU}
U = \left[
\begin{matrix}
1      & 0        &    & \\
\ell^1_1 & 1/2      & 0   & \\
\ell^2_2 & \ell^1_2/2 & 1/3  & \ddots \\
       & \ddots   & \ddots & \ddots
\end{matrix}
\right] ,
\end{equation}
and
\begin{equation*}
\mathbb S \, \mathcal P = A \,
\left[
\begin{matrix}
1      \\
\mathbb S \, \vartheta_1(t) \\
\mathbb S \, \vartheta_2(t) \\
\vdots
\end{matrix}
\right]
=
A \, \left[
\begin{matrix}
1      & 0        &    & \\
g_1 & 1      & 0   & \\
0 & g_2 & 1  & \ddots \\
       & \ddots   & \ddots & \ddots
\end{matrix}
\right]
\,
\left[
\begin{matrix}
1      \\
 \vartheta_1 \\
 \vartheta_2 \\
\vdots
\end{matrix}
\right] ,
\end{equation*}
or
\begin{equation}
\label{eq:gern}
A \, G = U \, \tilde{A} \, ,
 \ \ \mbox{ i.e. } \ \ \
A \, G \, \tilde{A}^{-1} = U \, ,
\end{equation}
where
\begin{equation*}
G = \left[
\begin{matrix}
1      & 0        &    & \\
g_1 & 1      & 0   & \\
0 & g_2 & 1  & \ddots \\
       & \ddots   & \ddots & \ddots
\end{matrix}
\right] .
\end{equation*}
By applying the divided-difference operator $\mathbb D$ to the three-term recurrence relation \eqref{eq:ttrrpp} satisfied by the sequence $\{ P_{n} \}_{n \geq 0}$ we obtain
\begin{equation*}
\mathbb S \, P_n = \mathbb D \, P_{n+1} + \beta_n \, \mathbb D \, P_n + \gamma_n \, \mathbb D \, P_{n-1} - \mathbb S \, \mu(t) \, \mathbb D \, P_n ,
\end{equation*}
which in matrix form can be expressed as
\begin{equation*}
\mathbb S \, \mathcal P = L \, D \, \mathcal P^{\prime} - D \, \mathbb S \, \mu(t) \, \mathcal P^{\prime} ,
\end{equation*}
with
\begin{equation*}
\mathbb S \, \mu(t) = \mu(t) + \frac{1}{4}.
\end{equation*}
{}From the recurrence relation \eqref{eq:ttrrdp} for the sequence of divided-differences $\{ P_n^{\prime}\}_{n \geq 1}$ we have
\begin{equation*}
\mathbb S \, \mathcal P = L \, D \, \mathcal P^{\prime} - D \, \big (M + \frac{1}{4} \, I \big) \, \mathcal P^{\prime} ,
\end{equation*}
or, by~\eqref{eq:ger}
\begin{equation*}
U \, \mathcal P^{\prime} = L \, D \, \mathcal P^{\prime} - D \, \big (M + \frac{1}{4} \, I \big) \, \mathcal P^{\prime} ,
\end{equation*}
i.e.
\begin{equation} \label{eq:noo}
U = L \, D - D \, \big (M + \frac{1}{4} \, I \big).
\end{equation}

\medskip
Therefore, we obtain the Geronimus-type characterization of classical orthogonal polynomials on quadratic lattices in matrix form as
\begin{proposition}
$\{P_{n}\}_{n \geq 0}$ is a sequence of classical orthogonal polynomials on the quadratic lattice $\mu(t)$ defined in \eqref{eq:lattice} if and only if \eqref{eq:noo} holds true, where $U$, $L$, $D$, and $M$ are defined in \eqref{eq:matrixU}, \eqref{eq:matrixL}, \eqref{eq:matrixD}, and \eqref{eq:matrixM}, respectively and $I$ denotes the identity matrix.
\end{proposition}

\subsection{A new characterization of classical orthogonal polynomials on quadratic lattices}

Let us recall Lemma \ref{eq:doc48)} as well as the following properties
\begin{align}
&L U = L^2 D - L D M -  \displaystyle\frac{1}{4}L D, \label{eq:doc45)} \\
&U M = L D M - D M^2 -  \displaystyle\frac{1}{4}D M, \label{eq:doc46)} \\
&A J A^{-1} D = D, \quad \tilde{D} D = I, \quad \text{and} \quad D\tilde{D} = J, \label{eq:doc47)}
\end{align}
where $J$ has been defined in \eqref{eq:defJJ} and
\begin{equation}\label{eq:matrixG}
G = \left[
\begin{matrix}
1           & 0               &        &  \\
g_1 & 1           & 0      & \\
0 & g_2 & 1    & \ddots \\
              & \ddots     & \ddots & \ddots
\end{matrix} \right]= I + E, \qquad
E = \left[
\begin{matrix}
0           & 0               &        &  \\
g_1 & 0           & 0      & \\
0 & g_2 &  0   & \ddots \\
              & \ddots     & \ddots & \ddots
\end{matrix} \right].
\end{equation}
We have
\[ 
\displaystyle L^2 \, D - 2 L \, D \, M + D\, M^2 + \frac{1}{4}(D \, M - L \, D) = A \Big( {\pmb{X^1}} \, E - E\, {\pmb{X^1}} \Big) {\tilde A}^{-1} \, ,
\]
with
\[
{\pmb{X^1}} E - E {\pmb{X^1}}= \left[
\begin{matrix}
 g_1 & 0 & 0 & 0 &  \\
  g_1 (f_1 - f_0) & g_2 - g_1  & 0 & 0 &   \ddots \\
  0 & g_2 (f_2 - f_1)  &  g_3 - g_2 & 0 & \ddots \\
    \vdots  &  \ddots & \quad \quad  \ddots & \quad \quad \ddots \\
\end{matrix} \right].
\]
Since
\begin{equation*}
f_{n+1} - f_n = \frac{1}{2}(n+1), \quad  f_{n+1} + f_n = \displaystyle\frac{1}{2} (n+1)^2 + 2( c_3 +
\displaystyle\frac{1}{16} -\displaystyle\frac{c_2^2}{4}),
\end{equation*}
where $f_{n}$ are given in \eqref{eq:fnfinal}, as well as
\begin{equation}\label{eq:doc4i)}
 g_{n+1} - g_n = \displaystyle\frac{1}{4}(4n+1) = \displaystyle\frac{1}{2}(2n+1) - \displaystyle\frac{1}{4},
 \end{equation}
we have
\begin{align}\label{eq:doc4ii)}
   & g_{n+1} (f_{n+1} - f_n ) \\
 \nonumber
    &  \phantom{olao}
= \displaystyle\frac{1}{2} (n+1) (f_{n+1}+ f_n )- \displaystyle\frac{1}{4} (n+1) (f_{n+1} - f_n )
             -  \big(c_3 + \displaystyle\frac{1}{16} - \displaystyle\frac{c_2^2}{4}\big) (n+1) \,.
\end{align}
Thus,
\begin{equation}\label{eq:doc4starstar}
{\pmb{X^1}} E - E {\pmb{X^1}} = \displaystyle\frac{1}{2}({\pmb{X^1}} D + D {\pmb{X^1}}) - \displaystyle\frac{1}{4}({\pmb{X^1}} D - D {\pmb{X^1}}) - \big(c_3 + \displaystyle\frac{1}{16} -
\displaystyle\frac{c_2^2}{4}\big)D,
\end{equation}
and from \eqref{eq:doc48)} we obtain
\begin{align*}
A{\pmb{X^1}}D\tilde{A}^{-1} = A{\pmb{X^1}} A^{-1} (AD\tilde{A}^{-1}) = LD, \\
AD{\pmb{X^1}}\tilde{A}^{-1} = (AD\tilde{A}^{-1}) \tilde{A} {\pmb{X^1}} \tilde{A}^{-1} = DM.
\end{align*}
By using \eqref{eq:doc4starstar} we obtain
\[
A({\pmb{X^1}} E - E{\pmb{X^1}})\tilde{A}^{-1} = \displaystyle\frac{1}{2}(LD + DM)- \frac{1}{4} (LD - DM) -  (c_3 + \displaystyle\frac{1}{16} -\displaystyle\frac{c_2^2}{4})D.
\]
Therefore,
\begin{multline*}
L^2 D - 2LDM + DM^2 - \displaystyle\frac{1}{4} (LD - DM) \\
= \displaystyle\frac{1}{2} (LD + DM)- \displaystyle\frac{1}{4}(LD - DM)  -  \big(c_3 + \displaystyle\frac{1}{16} - \displaystyle\frac{c_2^2}{4}\big)D.
\end{multline*}
Hence, the result follows.

\medskip We are now in conditions to state a new characterization of classical orthogonal polynomials on quadratic lattices, which is the extension of previous works \cite{MR3028602,MR3134272}:
\begin{theorem}
$\{P_{n} \}_{n \geq 0}$ is a sequence of classical orthogonal polynomials on the quadratic lattice $\mu(t)$ defined in \eqref{eq:lattice} if and only if
\begin{equation}\label{eq:newchar}
L^2 D - 2LDM + DM^2 - \displaystyle\frac{1}{2}(LD + DM)+  \big(c_3 + \displaystyle\frac{1}{16} - \displaystyle\frac{c_2^2}{4}\big)D = 0,
\end{equation}
holds true, where the matrices $L$, $D$, and $M$ are defined in \eqref{eq:matrixL}, \eqref{eq:matrixD}, and \eqref{eq:matrixM}, respectively, and the quadratic lattice $\mu(t)$ depends on the constants $c_{2}$ and $c_{3}$.
\end{theorem}

\subsection{Tricomi's characterization}
Classical orthogonal polynomials on quadratic lattices can be also characterized in terms of a structure relation of the form (Tricomi's characterization) \cite{MR0027919}
\begin{equation*}
\phi \, \mathbb D \, P_n = g_{n}^0 \, \mathbb S \, P_{n+1} + g_{n}^1 \, \mathbb S \, P_{n} + g_{n}^2 \, \mathbb S \, P_{n-1} ,
\end{equation*}
where $\phi$ is a polynomial of at most degree $2$ in the lattice $\mu(t)$, which can be written in matrix form as
\begin{equation*}
\phi \, \mathcal P^{\prime} = W \, \mathbb S \, \mathcal P ,
\end{equation*}
with
\begin{equation}\label{eq:matrixW}
W =
\left[ \begin{matrix}
g_1^2  &  g_1^1  &  g_1^0  &   0   & \\
0      & g_2^2  &  g_2^2  & g_2^0 & \ddots \\
       & \ddots   &  \ddots  & \ddots  & \ddots
\end{matrix} \right] .
\end{equation}
Thus,
\begin{equation*}
 \tilde A \, \phi (X+ \operatorname{diag})
:= \tilde A \, \phi ({\pmb{X^1}})
= W \, A \, G.
\end{equation*}
Notice that
\begin{equation*} \big( {\pmb{X^1}} \big)^2= \left[
\begin{matrix}
 f_0^2 & f_0 + f_1    & 1       & 0    & \\
 0    & f_1^2       & f_1+ f_2 & 1    & \ddots \\
     & \ddots      & \ddots   & \ddots & \ddots
\end{matrix}
\right] \, .
\end{equation*}
Therefore
\begin{equation}\label{eq:ctricomi}
\tilde A \, \phi ({\pmb{X^1}})= W \, A \, G \, .
\end{equation}
As
$ A \, G = U \, \tilde{A} $,
multiplying the first equation by $U$ (left) and using the second identity we have
\begin{equation*}
U \, W \, A \, G = A \, G \, \phi ({\pmb{X^1}})
 \ \ \mbox{ or equivalently } \ \ \
U \, W = (A \, G) \, \phi ({\pmb{X^1}}) \, (A \, G)^{-1} \, ,
\end{equation*}
i.e.
\begin{equation}\label{eq:cortricomi1}
U \, W = \phi \big( (A \, G) \, {\pmb{X^1}} \, (A \, G)^{-1} \big) \, .
\end{equation}
Multiplying now first by $W$ (left) the second equation and applying the first equation we obtain
\begin{equation*}
W \, U \, \tilde A = \tilde A \, \phi \big( {\pmb{X^1}} \big)
  \ \ \mbox{ or } \ \ \
W \, U = \tilde A \, \phi \big( {\pmb{X^1}} \big) \, \tilde A^{-1}
\end{equation*}
and
\begin{equation}\label{eq:cortricomi2}
W \, U =  \phi \big( \tilde A \, {\pmb{X^1}} \, \tilde A^{-1} \big) \, .
\end{equation}

\medskip
Thus, we can rewrite the Tricomi-type characterization of classical orthogonal polynomials on quadratic lattices by using the matrix approach as
\begin{proposition}
$\{P_{n} \}_{n \geq 0}$ is a sequence of classical orthogonal polynomials on the quadratic lattice $\mu(t)$ defined in \eqref{eq:lattice} if and only if \eqref{eq:cortricomi1} and \eqref{eq:cortricomi2} hold true, where the matrices $W$, $U$, $\tilde{A}$, and ${\pmb{X}}^{1}$ are defined in \eqref{eq:matrixW}, \eqref{eq:matrixU}, \eqref{eq:matrixAtilde}, and \eqref{eq:x1}, respectively.
\end{proposition}

\subsection{Bochner's characterization} Classical orthogonal polynomials on quadratic lattices are solution of a second-order divided-difference equation (Bochner's characterization) \cite{MR2836867}
\begin{equation}\label{eq:sodde}
\phi \, \mathbb D^2 \, P_n + \psi \, \mathbb S \, \mathbb D \, P_n = \lambda_n \, P_n
\end{equation}
where
\begin{gather}
\phi \equiv \phi(\mu(t))=a_{0} (\mu(t))^{2} + a_{1} \mu(t) + a_{2}, \label{eq:defphi} \\
\psi \equiv \psi(\mu(t)) = b_{0} \mu(t) + b_{1}, \label{eq:defpsi}
\end{gather}
are polynomials of at most degree $2$ and $1$ in the lattice $\mu(t)$. We can express the above characterization in matrix form as
\begin{equation*}
 A \, D^2 \, \phi ({\pmb{X^1}}) + A \, D \,G\, \psi ({\pmb{X^1}}) = \Lambda \, A \,,
\end{equation*}
with
\begin{equation}\label{eq:eigenvalue}
\lambda _n = n \, \big( (n-1) a_{0} + b_{0} \big) \, ,
\end{equation}
and
\begin{equation}\label{eq:matrixLambda}
\Lambda = \operatorname{diag} \{ \lambda_0 , \lambda_1 , \lambda_3 , \ldots \} \, .
\end{equation}

The Bochner equation can be written as an algebraic Sylvester equation in $A$, namely
\begin{equation}\label{eq:cbochner}
A \, \big( D^2 \, \phi ({\pmb{X^1}}) + D \, G \, \psi ({\pmb{X^1}}) \big) = \Lambda \, A \,.
\end{equation}
Therefore, given $\phi$ and $\psi$ the matrix $A$ is determined if and only if the point spectra of $\Lambda$ and $D^2 \, \phi ({\pmb{X^1}}) + D\, G \, \psi ({\pmb{X^1}}) $ be disjoint i.e. the given matrices do not have common eigenvalues.

\medskip
As a conclusion, we have the following Bochner-type characterization of classical orthogonal polynomials on quadratic lattices as
\begin{proposition}
$\{P_{n} \}_{n \geq 0}$ is a sequence of classical orthogonal polynomials on the quadratic lattice $\mu(t)$ defined in \eqref{eq:lattice} if and only if \eqref{eq:cbochner} holds true, assuming that $\lambda_{n}\neq \lambda_{m}$ for any $n,m =0,1,2,\dots$, $n \neq m$, where the matrices $A$, $D$, ${\pmb{X}}^{1}$, $G$, and $\Lambda$ are defined in \eqref{eq:matrixA}, \eqref{eq:matrixD}, \eqref{eq:x1}, \eqref{eq:matrixG}, and \eqref{eq:matrixLambda}, respectively.
\end{proposition}

\section{Solving the Bochner-type equation}\label{section:4}
In 1942 and 1943 \cite{branquinho} Vicente Gon\c{c}alves published two papers \cite{goncalves1,goncalves2} about classical orthogonal polynomials (Hermite, Jacobi, Laguerre and Bessel), proving the following result. Let $\sigma(x)=a_{0} x^{2}+a_{1}x+a_{2}$, $\tau(x)=b_{0}x+b_{1}$, and $\lambda_{n}=n((n-1)a_{0}+b_{0})$. Assuming that for each $n$ there exists a unique monic polynomial solution of the the equation, each element of the monic polynomial sequence $\{y_{n}\}_{n \geq 0}$ satisfies
\begin{equation*}
\sigma(x) y^{\prime\prime}(x) + \tau(x) y^{\prime}(x) - \lambda_{n} y=0, \qquad (n=0,1,2,\dots)
\end{equation*}
then the monic polynomial sequence $\{y_{n}\}_{n \geq 0}$ satisfies the above equation
if and only if $\{y_{n}\}_{n \geq 0}$ satisfies a three-term recurrence relation
\begin{equation*}
x y_{n}=y_{n+1} + \beta_{n} y_{n} + \gamma_{n} y_{n-1}, \quad n \geq 1,
\end{equation*}
where the two sequences of real numbers $\{\beta_{n}\}_{n \geq 0}$ and $\{\gamma_{n}\}_{n \geq 1}$ are fully determined by the constants $a_{0}$, $a_{1}$, $a_{2}$, $b_{0}$, and $b_{1}$.

\medskip
Next we reinterpret the above result for quadratic lattices. From \eqref{eq:sodde} let us introduce the Bochner-type operator
\begin{equation}\label{eq:bochnernul}
\operatorname L_n = \phi \, \mathbb D^2 + \psi \, \mathbb S \, \mathbb D - \lambda_n \, I,
\end{equation}
where $\phi$ and $\psi$ are polynomials in the lattice $\mu(t)$ defined in \eqref{eq:defphi} and \eqref{eq:defpsi}. We shall assume that $\operatorname L_n$ has for each nonnegative integer $n$ a unique monic polynomial solution of degree exactly $n$ in the quadratic lattice $\mu(t)$, denoted by $P_{n}\equiv P_{n}(\mu(t))$, i.e. $P_{n}= \vartheta_n (t) +p_{1,n} \vartheta_{n-1}(t) + p_{2,n} \vartheta_{n-2}(t)+$ terms of lower degree and
\begin{equation*}
\operatorname L_n \big( P_n \big) = 0, \quad n = 0,1, \ldots.
\end{equation*}
Notice that $\operatorname L_n$ acting on a polynomial $g$ of degree $n$ in the lattice $\mu(t)$ gives a new polynomial of degree at most $n$ in the lattice $\mu(t)$. Let us recall the expression \eqref{eq:1} of the polynomial $P_{n}$ in terms of the basis $\{\vartheta_{n}(t)\}$. First, we state a result for the unicity of monic polynomial solution of the Bochner-type equation \eqref{eq:sodde}.
\begin{lemma}
For each $n$, the unicity of monic polynomial solution of the Bochner-type equation \eqref{eq:sodde} is equivalent to
\begin{enumerate}
\item  $\lambda_{j}=\lambda_{n}$ has $j=n$ as unique solution in  $\mathbb N$;
\item $\lambda_{k} \not = 0$, $k = 0,1, \ldots, n-1$.
\end{enumerate}
\end{lemma}

\begin{lemma}
There exists a sequence $\{\beta_n\}_{n \in {\mathbb{N}}}$ such that the polynomial
\begin{equation}\label{eq:polun}
U_n(\mu(t)) = \operatorname L_{n+1} \big( (\mu(t) - \beta_n) P_n \big) ,
\end{equation}
has degree $n-1$ in the lattice $\mu(t)$, for each $n \in \mathbb N$. Moreover
\begin{equation}\label{eq:eqbeta}
 \beta_n  = p_{1,n} + f_{n} + \displaystyle\frac{k_{1,n+1}}{\lambda_{n} - \lambda_{n+1}},
\end{equation}
and $U_{n}(\mu(t))= t_n \, \vartheta_{n-1} + \cdots$ where
\begin{equation}\label{eq:eqtn}
t_n = k_{2,n+1} +  ( f_{n}  + p_{1,n} - \beta_n )k_{1,n} + (p_{1,n} f_{n-1} + p_{2,n} - \beta_n p_{1,n} )(\lambda_{n-1} - \lambda_{n+1}).
\end{equation}
\end{lemma}
\begin{proof}
{}From \eqref{eq:1} we have
\begin{multline*}
\big( \mu(t) - \beta_n \big )\,P_{n}(\mu(t)) \\
=  \vartheta_{n+1}(t)+ ( f_{n}  + p_{1,n} - \beta_n )\,\vartheta_{n}(t) + (p_{1,n}
f_{n-1} + p_{2,n} - \beta_n p_{1,n} )\,\vartheta_{n-1}(t) + \cdots.
\end{multline*}
By using \eqref{eq:divdif} and \eqref{eq:18},
\begin{equation}\label{eq:neweq1}
\operatorname L_{n}\big(\vartheta_j (t) \big) =k_{0,j}\,\vartheta_{j}(t) + k_{1,j}\,\vartheta_{j-1}(t) +k_{2,j}\,\vartheta_{j-2}(t),
\end{equation}
where
\begin{equation}\label{eq:neweq2}
\begin{cases}
k_{0,j} = a_{0} j(j-1) + b_{0} j - \lambda_{n}, \\
k_{1,j} = a_{0} j(j-1)(f_{j-1} + f_{j-2}) + b_{0} j f_{j-1} + a_{1} j(j-1) + b_{0} j g_{j-1} + b_{1} j, \\
k_{2,j} = a_{0} j(j-1) f_{j-2}^2  + a_{1} j(j-1) f_{j-2} + b_{0} j g_{j-1} f_{j-2} + a_{2} j (j - 1) + b_{1} j g_{j-1}.
\end{cases}
\end{equation}
Therefore,
\begin{align*}
\operatorname L_{n+1}\big(\vartheta_{n+1} (t) \big) &=  [ a_{0} n(n+1) + b_{0}(n+1) - \lambda_{n+1}]\,\vartheta_{n+1}(t) +
k_{1,n+1}\,\vartheta_{n}(t) + k_{2,n+1}\,\vartheta_{n-1}(t), \\
\operatorname L_{n+1}\big(\vartheta_{n} (t) \big) &=  [ a_{0} n(n-1) + b_{0} n - \lambda_{n+1}]\,\vartheta_{n}(t) +
k_{1,n}\,\vartheta_{n-1}(t) + k_{2,n}\,\vartheta_{n-2}(t), \\
\operatorname L_{n+1}\big(\vartheta_{n-1} (t) \big) &=  [ a_{0} (n-1)(n-2) +  b_{0} (n-1) - \lambda_{n+1}]\,\vartheta_{n-1}(t) \\
&+k_{1,n-1}\,\vartheta_{n-2}(t) + k_{2,n-1}\,\vartheta_{n-3}(t).
\end{align*}
As a consequence,
\begin{multline*}
U_n = \big [ a_{0} n(n+1) + b_{0}(n+1) - \lambda_{n+1}]\,\vartheta_{n+1}(t) \\
+ \big[k_{1,n+1} + ( f_{n}  + p_{1,n} - \beta_n )(a_{0} n(n-1) + b_{0} n - \lambda_{n+1})  \big]\,\vartheta_{n}(t)\\ 
+ \big[k_{2,n+1} +  ( f_{n}  + p_{1,n} - \beta_n )k_{1,n}  + (p_{1,n} f_{n-1} + p_{2,n} \\
- \beta_n p_{1,n} )(a_{0} (n-1)(n-2) + b_{0} (n-1) - \lambda_{n+1}) \big]\,\vartheta_{n-1}(t) + \cdots.
\end{multline*}
Thus, the coefficient in $\vartheta_{n+1}$ is zero since $\lambda_{n+1} = a_{0} n(n+1) + b_{0}(n+1)$. Moreover, in order that $U_{n}(\mu(t))$ in \eqref{eq:polun} be a polynomial of degree $n-1$ in $\mu(t)$ we get \eqref{eq:eqbeta} as well as $\lambda_{n+1} \neq \lambda_{n}$. Finally, we also obtain that the coefficient in $\vartheta_{n-1}$ in \eqref{eq:polun} is given by \eqref{eq:eqtn}.
\end{proof}

In order to continue with the method of Vicente Gon\c{c}alves for quadratic lattices, and since the proofs are rather technical, we shall first state the results, while the complete proofs are detailed later.
\begin{lemma}\label{lemma:41}
For each natural number $n$ we have
$ \operatorname L_{n-1} \big( U_n(\mu(t) \big) = 0 \, $,
where $U_{n}(\mu(t))$ is defined in \eqref{eq:polun}.
\end{lemma}

\medskip
{}From the unicity of solution of Bochner's equation, there exists a constant $t_{n}$ such that 
\begin{equation}\label{numerar}
\displaystyle U_n = t_n \, P_{n-1} \, .
\end{equation}

\medskip
\begin{lemma}\label{lemma:43}
Let $P_{n}$ be the unique monic polynomial solution of degree $n$ in the quadratic lattice $\mu(t)$ of the Bochner equation \eqref{eq:sodde}. Then, there exist sequences $\{\beta_{n}\}_{n \geq 0}$ and $\{\gamma_{n}\}_{n \geq 1}$ such that the following three-term recurrence relation holds
\begin{equation}\label{eq:ttrr}
P_{n+1} = (\mu(t) - \beta_n) \, P_n - \gamma_n \, P_{n-1}.
\end{equation}
More precisely, $\beta_{n}$ is given in \eqref{eq:eqbeta} and
\begin{equation}\label{eq:eqgamma}
\gamma_{n}=\frac{t_{n}}{\lambda_{n-1}-\lambda_{n+1}}.
\end{equation}
\end{lemma}

As a summary of the previous results we have
\begin{theorem}\label{theorem:44}
Let $P_{n}$ be the monic polynomial solution of degree $n$ in the quadratic lattice $\mu(t)$ of the second-order linear divided-difference equation \eqref{eq:sodde}, where the polynomials $\phi$ and $\psi$ are given in \eqref{eq:defphi} and \eqref{eq:defpsi}, respectively, and the eigenvalue $\lambda_{n}$ is given in~\eqref{eq:eigenvalue}. Then, the coefficients $\beta_{n}$ and $\gamma_{n}$ of the three-term recurrence relation \eqref{eq:ttrr} satisfied by the sequence $\{P_{n}\}_{n \geq 0}$ are given by
\begin{align}
\beta_{n}& = p_{1,n}-p_{1,n+1}+f_n, \label{eq:betanfinal} \\
\gamma_n & = p_{1,n} \left(f_{n-1}-\beta_{n}\right)+p_{2,n}-p_{2,n+1}, \label{eq:gammanfinal}
\end{align}
where
\begin{align}
p_{1,n}& =-\frac{n \left(a (n-1) \left(f_{n-2}+f_{n-1}\right)+b (n-1)+r \left(f_{n-1}+ g_{n-1}\right)+s\right)}{\lambda _{n-1}-\lambda _n}, \label{eq:p1nfinal} \\
\label{eq:p2nfinal}
p_{2,n}& =-\frac{1}{\lambda _{n-2}-\lambda _n}
\left \{ (n-1) \left(p_{1,n} \left(a (n-2) \left(f_{n-3}+f_{n-2}\right)+b (n-2) \right. \right. \right.   \\
 & \phantom{olaola} \left. \left. \left. +r \left(f_{n-2}+g_{n-2}\right)+s\right)+n \left(f_{n-2} \left(a f_{n-2}+b\right)+c\right)\right)+n g_{n-1} \left(r f_{n-2}+s\right) \right \}, \nonumber \\
\lambda_{n}&=n (a (n - 1) + r), \label{eq:lambdanfinal}
\end{align}
and the coefficients $f_{n}$ and $g_{n}$ are given in \eqref{eq:fnfinal} and \eqref{eq:gnfinal}, respectively.
\end{theorem}

\begin{example}
As an example of application of the previous results, let us recall that monic Racah polynomials can be defined in terms of hypergeometric series as \cite[page 190]{MR2656096}
\begin{multline*}
r_{n}(\alpha,\beta,\gamma,\delta;t)=r_{n}(t)=\frac{(\alpha+1)_{n}\,(\beta+\delta+1)_{n}\,(\gamma+1)_{n}}{(n+\alpha+\beta+1)_{n}} \\
\times
\hyper{4}{3}{-n,n+\alpha+\beta+1,-t,t+\gamma+\delta+1}{\alpha+1,\beta+\delta+1,\gamma+1}{1},\quad
n=0,1,\ldots,N,
\end{multline*}
where $r_{n}(\alpha,\beta,\gamma,\delta;t)$ is a polynomial of degree $n$ in the quadratic lattice
\begin{equation*}
\mu(t)=t(t+\gamma+\delta+1).
\end{equation*}
Racah polynomials satisfy a second-order linear divided-difference equation which can be written as a Bochner-type equation of the form \eqref{eq:sodde} where $\phi$ is the polynomial of degree two in the lattice $\mu(t)$ given by
\begin{multline*}
\phi(\mu(t))=-(\mu(t))^{2} + \frac{1}{2} (-\alpha (2 \beta +\delta +\gamma +3)+\beta (\delta -\gamma -3)-2 (\delta \gamma +\delta +\gamma +2)) \mu(t) \\
-\frac{1}{2} (\alpha +1) (\gamma +1) (\beta +\delta +1) (\delta +\gamma +1),
\end{multline*}
$\tau$ is the polynomial of degree one in the lattice $\mu(t)$ given by
\begin{equation*}
\tau(\mu(t))=-(\alpha +\beta +2) \mu(t)-(\alpha +1) (\gamma +1) (\beta+\delta +1),
\end{equation*}
and the eigenvalues $\lambda_{n}$ are given by
\begin{equation*}
\lambda_{n}=-n (\alpha +\beta +n+1).
\end{equation*}
If we apply Theorem \ref{theorem:44} we obtain exactly the coefficients of the three-term recurrence relation \cite[Eq. (9.2.4)]{MR2656096}. In a similar way, Theorem \ref{theorem:44} can be applied to obtain the coefficients of the three-term recurrence relation satisfied by any sequence of monic orthogonal polynomials solution of a Bochner-type equation on a quadratic lattice \eqref{eq:sodde}, assuming that the equation has a unique monic polynomial solution for each positive integer $n$.
\end{example}

\medskip
\begin{proof}[Proof of Lemma \ref{lemma:41}]
We shall need the following relations
\begin{enumerate}
    \item [a)] $\mathbb D [f\,g] = \mathbb S f \, \mathbb D g + \mathbb D f \,\mathbb S g $,
    \item [b)] $\mathbb S [f\,g] =  m_2(t) \mathbb D f \,\mathbb D g + \mathbb S f \, \mathbb S g
    $, with $m_2(t) = \mu(t)+ \delta_x $,
    \item [c)] $\mathbb S [\mu(t)] =  \displaystyle \mu(t)+ 1/{4} $,
    \item[d)] $\mathbb D \mathbb S f = \mathbb S \mathbb D f + m_1\,\mathbb D^2 f$, with $m_1 = 1/2$,
    \item [e)] $ \mathbb S^2 f =   m_1 \,\mathbb S \mathbb D f + m_2(t)\,\mathbb D^2 f + f$.
\end{enumerate}
{}From the definition of the linear operator $L_{n}$  and the polynomial $U_{n}(\mu(t))$ we have
\begin{eqnarray*}
U_n & = & \operatorname L_{n+1} \big( (\mu(t) - \beta_n) P_n \big) \\
 & = & \phi \, \mathbb D^2 \big( (\mu(t) - \beta_n) P_n \big) + \psi \, \mathbb S \, \mathbb D \big( (\mu(t) - \beta_n) P_n \big) - \lambda_{n+1} \, \big( (\mu(t) - \beta_n) P_n \big) \\
 & = & \phi \, \mathbb D \big( (\mu(t)\mathbb  + \frac{1}{4}- \beta_n) \mathbb D P_n + \mathbb S P_n \big)
+ \psi \, \mathbb S  \big( \mu(t)\mathbb  + \frac{1}{4}- \beta_n) \mathbb D P_n + \mathbb S P_n \big) \\
 & & \phantom{olaolaolaolaolaolaolaolaolaolaolaolaolaolaolaola}- \lambda_{n+1} \, \big( (\mu(t) - \beta_n) P_n \big)
 \\
 & = & \phi \big( (\mu(t)- \beta_n + 1)\mathbb D^2 P_n + 2\mathbb S \, \mathbb D P_n \big) + \psi \big( 2( \mu(t)+ \delta_x)
\mathbb D^2 P_n
 \\
 & & \phantom{olaolaolaolaolaola}+ (\mu(t) - \beta_n + 1 ) \mathbb S\, \mathbb D P_n + P_n \big) - \lambda_{n+1} \, \big( (\mu(t) - \beta_n) P_n \big) \\
 & = &  (\mu(t)- \beta_n + 1)\big(\phi \mathbb D^2 P_n + \psi \mathbb S\, \mathbb D P_n \big) + 2 \phi \mathbb S \, \mathbb D P_n \\
 & & \phantom{olaolaolaolaolaola} + \psi \big( 2( \mu(t)+ \delta_x) \mathbb D^2 P_n + P_n \big) - \lambda_{n+1} \, \big( (\mu(t) - \beta_n) P_n \big)
 \\
 & = & (\mu(t) - \beta_n + 1 )\lambda_n P_n + 2 \phi \mathbb S \,\mathbb D P_n
+ \psi \big( 2( \mu(t)+ \delta_x) \mathbb D^2 P_n + P_n \big) \\
 & & \phantom{olaolaolaolaolaolaolaolaolaolaolaolaolaolaolaola} - \lambda_{n+1} \, \big( (\mu(t) - \beta_n) P_n \big)
 \\
 & = & 2 \phi \mathbb S \, \mathbb D P_n + \psi P_n + (\mu(t) - \beta_n )\lambda_n P_n - \lambda_{n+1}  (\mu(t) - \beta_n) P_n
+  \lambda_n P_n \\
 & & \phantom{olaolaolaolaolaolaolaolaolaolaolaolaolaolaolaola} + 2 \psi ( \mu(t)+ \delta_x) \mathbb D^2 P_n \\
 & = & 2 \phi \mathbb S \, \mathbb D P_n + \psi P_n + (\lambda_n -\lambda_{n+1} )(\mu(t) - \beta_n ) P_n
+ \lambda_n P_n + 2\psi(\mu(t)+ \delta_x) \mathbb D^2 P_n.
\end{eqnarray*}
As a consequence,
\begin{multline*}
\operatorname L_{n-1} \big( U_n(\mu(t) \big) \\
= \operatorname L_{n-1} \big[ 2 \phi \mathbb S \, \mathbb D P_n + \psi P_n + (\lambda_n - \lambda_{n+1} )(\mu(t) - \beta_n ) P_n +  \lambda_n P_n + 2m_2(t)\,\psi \mathbb D^2 P_n\big] .
\end{multline*}
We shall now obtain a number of properties which shall be used in the proof. First,
\begin{multline}\label{eq:I)}
 \operatorname L_{n+1} \big[(\lambda_n - \lambda_{n+1} ) (\mu(t)
- \beta_n) P_n \big] \\
= 2(\lambda_n - \lambda_{n+1} ) \phi \mathbb
S \, \mathbb D P_n  + (\lambda_n - \lambda_{n+1} )(\lambda_n -
\lambda_{n-1} )(\mu(t) - \beta_n ) P_n \\ + (\lambda_n -
\lambda_{n+1} )\psi P_n + \lambda_n(\lambda_n - \lambda_{n+1}
) P_n + 2(\lambda_n - \lambda_{n+1} ) m_2(t)\,\psi\, \mathbb D^2 P_n.
\end{multline}
Moreover,
\begin{equation}\label{eq:II)}
\operatorname L_{n-1} \big[\lambda_n P_n\big] =  \lambda_n(\lambda_n - \lambda_{n-1} ) P_n.
\end{equation}
We shall also need the following relations:
\begin{equation*}
\mathbb S^2 \mathbb D^2 P_n = m_2(t) \mathbb D^4 P_n + m_1\,\mathbb S \mathbb D^3 P_n + \mathbb D^2 P_n,
\end{equation*}
\begin{equation*}
\mathbb D \mathbb S \mathbb D^2 P_n = \mathbb S \mathbb D^3 P_n + m_1 \,\mathbb D^4 P_n,
\end{equation*}
\begin{equation*}
\mathbb D^2 \mathbb S \mathbb D P_n = \mathbb D \mathbb
S \mathbb D^2 P_n + m_1 \,\mathbb D^4 P_n
 = \mathbb S \mathbb
D^3 P_n + 2 m_1 \,\mathbb D^4 P_n,
\end{equation*}
\begin{equation*}
 \mathbb S \mathbb D \mathbb S \mathbb D P_n = \mathbb S^2
\mathbb D^2 P_n + m_1 \, \mathbb S \mathbb
D^3 P_n
= m_2(t) \mathbb D^4 P_n + 2 m_1\,\mathbb S \mathbb D^3
P_n + \mathbb D^2 P_n,
\end{equation*}
\begin{equation*}
 \mathbb D \mathbb S^2 \mathbb D P_n = \mathbb S \mathbb
D \mathbb S \mathbb D P_n + m_1 \,\mathbb D^2 \mathbb S \mathbb D
P_n \\
= (m_2(t) + 2 m_1 ^2)\,\mathbb D^4 P_n + 3 m_1\,\mathbb S \mathbb
D^3 P_n + \mathbb D^2 P_n,
\end{equation*}
\begin{multline*}
 \mathbb S^3 \mathbb D P_n = m_2(t) \,\mathbb D^2
\mathbb S \mathbb D P_n + m_1 \,\mathbb S \mathbb D \mathbb S
\mathbb D P_n + \mathbb S \mathbb D P_n \\
 =  (m_2(t) + 2 m_1 ^2)
\,\mathbb S \mathbb D^3 P_n  +  3 m_1 m_2(t)\,\mathbb D^4 P_n +
m_1 \mathbb D^2 P_n + \mathbb S \mathbb D P_n.
\end{multline*}
Also,
\begin{multline}\label{eq:III)}
\mathbb D^2 \big[\phi \mathbb S \, \mathbb D P_n \big] = \mathbb D^2 (\phi) \,\mathbb D P_n + \mathbb S \mathbb D (\phi) \,\mathbb D \mathbb S^2 \mathbb D P_n
+ \mathbb D \mathbb S (\phi) \,\mathbb S \mathbb D \mathbb S \mathbb D P_n + \mathbb S^2 (\phi) \,\mathbb D^2 \mathbb S \mathbb D P_n\\
= \mathbb D^2 (\phi) \big[(m_2(t) + 2 m_1 ^2) \,\mathbb S \mathbb
D^3 P_n  + 3 m_1 m_2(t)\,\mathbb D^4 P_n + m_1 \mathbb D^2 P_n +
\mathbb S \mathbb D P_n  \big]\\
+ \mathbb S \mathbb D (\phi) \big[(m_2(t) + 2 m_1 ^2)\,\mathbb D^4 P_n + 3 m_1\,\mathbb S \mathbb D^3 P_n + \mathbb D^2 P_n  \big]\\
+ \mathbb D \mathbb S (\phi) \big[m_2(t) \mathbb D^4 P_n + 2
m_1\,\mathbb S \mathbb D^3 P_n + \mathbb D^2 P_n \big]
+ \mathbb S^2 (\phi) \big[  \mathbb S \mathbb D^3 P_n + 2 m_1 \,\mathbb D^4 P_n \big] \\
= \big[ 3 m_1 m_2(t)\,\mathbb D^2(\phi) + m_2(t)\,\mathbb S \mathbb D (\phi) + 2 m_1^2 \mathbb S \mathbb D(\phi) + m_2(t) \mathbb D \mathbb S (\phi)
+ 2 m_1 \mathbb S^2 (\phi)  \big] \,\mathbb D^4 P_n\\
+ \big[m_2(t)\,\mathbb D^2(\phi) +2 m_1^2\, \mathbb D^2 (\phi) + 3
m_1 \mathbb S \mathbb D(\phi) + 2 m_1\mathbb D \mathbb S (\phi) +
\mathbb S^2 (\phi)   \big] \,\mathbb S \mathbb D^3 P_n \\ +
\big[m_1\,\mathbb D^2(\phi) + \mathbb S \mathbb D (\phi) + \mathbb
D \mathbb S (\phi)  \big]\,\mathbb D^2 P_n + \mathbb
D^2(\phi)\,\mathbb S \mathbb D P_n.
\end{multline}
Note that
\begin{multline*}
\mathbb D^2 \big[ \phi \,\mathbb D^2 P_n \big]  =  \mathbb D^2(\phi) \,\mathbb S^2 \mathbb D^2 P_n
+ \mathbb S \mathbb D (\phi) \,\mathbb D \mathbb S \mathbb D^2 P_n  + \mathbb D \mathbb S (\phi) \,\mathbb S \mathbb D^3 P_n + \mathbb S^2(\phi)\,\mathbb D^4 P_n \\
= \big[ m_2(t) \,\mathbb D^2(\phi) + m_1 \,\mathbb S \mathbb D
(\phi) + \mathbb S^2 (\phi)  \big]\,\mathbb D^4 P_n +
\big[m_1\,\mathbb D^2(\phi) + \mathbb S \mathbb D (\phi) \\
+ \mathbb D \mathbb S (\phi) \big] \,\mathbb S \mathbb D^3 P_n  + \mathbb D^2(\phi)\,\mathbb D^2 P_n,
\end{multline*}
as well as
\begin{multline*}
\mathbb S \mathbb D \big[\phi \,\mathbb D^2 P_n \big] = \mathbb S \mathbb D (\phi) \,\mathbb S^2 \mathbb D^2 P_n + \mathbb S^2(\phi)\,\mathbb S \mathbb D^3 P_n
+ m_2(t)\,\big[\mathbb D^2(\phi)\,\mathbb D \mathbb S \mathbb D^2 P_n  + \mathbb D \mathbb S (\phi) \,\mathbb D^4 P_n   \big]\\
= \big[ m_2(t) \,\mathbb S \mathbb D(\phi) + m_1 m_2(t) \, \mathbb D ^2(\phi) + m_2(t)\,\mathbb D \mathbb S (\phi)  \big]\,\mathbb D^4 P_n \\
 + \big[m_1\,\mathbb S \mathbb D(\phi) + \mathbb S^2 (\phi) + m_2(t)\,\mathbb D ^2 (\phi)  \big] \,\mathbb S \mathbb D^3 P_n
 + \mathbb S \mathbb D(\phi)\,\mathbb D^2 P_n .
\end{multline*}
As a consequence,
\begin{multline*}
\mathbb D^2 \big[\phi \mathbb S \, \mathbb D P_n \big] = 2 m_1 \,\mathbb D^2 \big[ \phi \,\mathbb D^2 P_n \big]
+ \mathbb S \mathbb D \big[\phi \,\mathbb D^2 P_n \big] + \big[m_1\,\mathbb D^2(\phi) + \mathbb S \mathbb D (\phi)
+ \mathbb D \mathbb S (\phi)  \big] \,\mathbb D^2 P_n \\
+ \mathbb D^2(\phi)\,\mathbb S \mathbb D P_n - 2m_1\,\mathbb D^2(\phi)\,\mathbb D^2 P_n - \mathbb S \mathbb D (\phi)\,\mathbb D^2 P_n\\
= 2 m_1 \,\mathbb D^2 \big[ \phi \,\mathbb D^2 P_n \big]  +
\mathbb S \mathbb D \big[\phi \,\mathbb D^2 P_n \big] + \mathbb S
\mathbb D (\phi) \,\mathbb D^2 P_n + \mathbb D^2(\phi)\,\mathbb S
\mathbb D P_n.
\end{multline*}
Therefore,
\begin{equation}\label{eq:1star}
 \mathbb D^2 \big[2 \phi \mathbb S \, \mathbb D P_n \big] =  4 m_1 \,\mathbb D^2 \big[ \phi \,\mathbb D^2 P_n \big]
 + 2\mathbb S \mathbb D \big[\phi \,\mathbb D^2 P_n \big] + 2\mathbb S \mathbb D (\phi) \,\mathbb D^2 P_n + 2\mathbb D^2(\phi)\,\mathbb S \mathbb D
 P_n.
\end{equation}
Furthermore,
\begin{multline}\label{eq:IV)}
\mathbb D^2 \big[\psi \,m_2(t)\,\mathbb D^2 P_n \big] = \mathbb
D^2 \big[ \psi \,m_2(t) \big] \,\mathbb S^2 \mathbb D^2 P_n +
\mathbb S \mathbb D \big[\psi \,m_2(t) \big]\, \mathbb D \mathbb S
\mathbb D^2 P_n \\ + \mathbb D \mathbb S \big[\psi \,m_2(t) \big]\,
\mathbb S \mathbb D^3 P_n + \mathbb S^2  \big[\psi \,m_2(t)
\big]\, \mathbb D^4 P_n.
\end{multline}
We have
\[
\mathbb D^2 \big[ \psi \,m_2(t) \big] = 4 m_1\,\mathbb S \mathbb D (\psi),
\]
\[
\mathbb S \mathbb D \big[\psi \,m_2(t)
\big] = m_2(t)\,\mathbb S \mathbb D (\psi)  + 2 m_1 ^2\,\mathbb S
\mathbb D (\psi) + 2 m_1\,\mathbb S^2 (\psi),
\]
\[
\mathbb D \mathbb S \big[\psi \,m_2(t) \big] = m_2(t)\,\mathbb S
\mathbb D (\psi)  + 6 m_1 ^2\,\mathbb S \mathbb D (\psi) + 2
m_1\,\mathbb S^2 (\psi),
\]
\[
\mathbb S^2 \big[\psi \,m_2(t) \big] = 4 m_1 m_2(t)\,\mathbb S
\mathbb D (\psi)  + 2 m_1 ^3\,\mathbb S \mathbb D (\psi) + (2
m_1^2  + m_2(t)\,\mathbb S^2 (\psi) .
\]
Then,
\begin{multline*}
\mathbb D^2 \big[\psi \,m_2(t)\,\mathbb D^2 P_n  \big] = 4 m_1\,\mathbb S \mathbb D (\psi) \, \big[  m_2(t) \mathbb D^4 P_n + m_1\,\mathbb S \mathbb D^3 P_n
+ \mathbb D^2 P_n \big]\\
+ \big[m_2(t)\,\mathbb S \mathbb D (\psi)  + 2 m_1 ^2\,\mathbb S
\mathbb D (\psi) + 2 m_1\,\mathbb S^2 (\psi)  \big]\big[\mathbb S
\mathbb D^3 P_n + m_1 \,\mathbb D^4 P_n  \big]
\\
+ \big[m_2(t)\,\mathbb S \mathbb D (\psi)  + 6 m_1 ^2\,\mathbb S \mathbb D (\psi) + 2 m_1\,\mathbb S^2 (\psi)  \big] \, \mathbb S \mathbb D^3 P_n\\
+ \big[4 m_1 m_2(t)\,\mathbb S \mathbb D (\psi)  + 2 m_1 ^3\,\mathbb S \mathbb D (\psi) + (2 m_1^2  + m_2(t))\,\mathbb S^2 (\psi)   \big] \, \mathbb D^4 P_n \\
= \big[9 m_1 m_2(t)\,\mathbb S \mathbb D (\psi)  + 4 m_1
^3\,\mathbb S \mathbb D (\psi) + 4 m_1^2 \,\mathbb S^2 (\psi)   +
m_2(t) \,\mathbb S^2 (\psi)  \big] \, \mathbb D^4 P_n\\
 +
\big[12m_1^2 \,\mathbb S \mathbb D (\psi) + 2 m_2 (t) \,\mathbb S
\mathbb D (\psi) + 4 m_1\,\mathbb S \mathbb D (\psi) + 4 m_1
\,\mathbb S^2 (\psi) \big] \, \mathbb S \mathbb D^3 P_n + 4
m_1\,\mathbb S \mathbb D (\psi)\,\mathbb D^2 P_n.
\end{multline*}
Notice that
\begin{multline*}
\mathbb D^2 \big[\psi \,\mathbb S \mathbb D P_n \big] = \mathbb D^2 (\psi)\,\mathbb S^2 \mathbb D^2 P_n
+ \mathbb S \mathbb D (\psi)\,\mathbb D \mathbb S^2 \mathbb D P_n + \mathbb D \mathbb S (\psi)\,\mathbb S \mathbb D \mathbb S \mathbb D P_n + \mathbb S^2 (\psi)\,\mathbb D^2 \mathbb S \mathbb D P_n \\
= \mathbb S \mathbb D (\psi)\,\big[ (m_2(t) + 2 m_1 ^2)\,\mathbb D^4 P_n + 3 m_1\,\mathbb S \mathbb D^3 P_n + \mathbb D^2 P_n \big] \\
+ \mathbb D \mathbb S (\psi)\,\big[m_2(t) \mathbb D^4 P_n + 2 m_1\,\mathbb S \mathbb D^3 P_n + \mathbb D^2 P_n  \big] + \mathbb S^2 (\psi)\,\big[ \mathbb S \mathbb D^3 P_n + 2 m_1 \,\mathbb D^4 P_n \big]\\
=  \big[ 2 m_2(t) \,\mathbb S \mathbb D (\psi) + 2 m_1 ^2
\,\mathbb S \mathbb D (\psi) + 2 m_1 \,\mathbb S^2 (\psi) \big]
\,\mathbb D^4 P_n \\
+ \big[5 m_1 \,\mathbb S \mathbb D (\psi) +
\mathbb S^2 (\psi)    \big] \,\mathbb S \mathbb D^3 P_n + 2
\mathbb S \mathbb D (\psi) \,\mathbb D^2 P_n,
\end{multline*}
and
\begin{multline*}
\mathbb S \mathbb D \big[\psi \,\mathbb S \mathbb D P_n \big] = \mathbb S \mathbb D (\psi)\,\mathbb S^3 \mathbb D P_n
+ \mathbb S^2 (\psi)\,\mathbb S \mathbb D \mathbb S \mathbb D P_n \\
+ m_2(t) \big[ \mathbb D^2 (\psi)\,\mathbb D \mathbb S^2 \mathbb D P_n + \mathbb D \mathbb S (\psi)\,\mathbb D^2 \mathbb S \mathbb D P_n \big]\\
= \mathbb S \mathbb D (\psi)\, \big[  (m_2(t) + 2 m_1 ^2) \,\mathbb S \mathbb D^3 P_n  +  3 m_1 m_2(t)\,\mathbb D^4 P_n +
m_1 \mathbb D^2 P_n  + \mathbb S \mathbb D P_n \big]\\
+ \mathbb S^2 (\psi)\,\big[m_2(t) \,\mathbb D^4 P_n + 2 m_1\,\mathbb S \mathbb D^3 P_n  + \mathbb D^2 P_n   \big]
+  m_2(t)\,\mathbb D \mathbb S (\psi)\,\big[ \mathbb S \mathbb D^3 P_n + 2 m_1 \,\mathbb D^4 P_n  \big]\\
= \big[5 m_1 m_2(t)\, \mathbb S \mathbb D (\psi) + m_2(t) \, \mathbb S^2 (\psi) \big] \,\mathbb D^4 P_n \\
+ \big[2 m_2(t) \,\mathbb S \mathbb D (\psi) + 2 m_1 ^2 \,\mathbb S \mathbb D (\psi) + 2 m_1\,   \mathbb S^2 (\psi)   \big] \,\mathbb S \mathbb D^3 P_n \\
+ \big[ m_1 \,\mathbb S \mathbb D (\psi) + \mathbb S^2 (\psi)
\big]\,\mathbb D^2 P_n + \mathbb S \mathbb D (\psi)\,\mathbb S
\mathbb D P_n.
\end{multline*}
Thus,
\begin{multline*}
\mathbb D^2 \big[\psi \,m_2(t)\,\mathbb D^2 P_n  \big] = 2
m_1\,\mathbb D^2 \big[\psi \,\mathbb S \mathbb D P_n \big] +
\mathbb S \mathbb D \big[\psi \,\mathbb S \mathbb D P_n \big] \\
- \big[ m_1\,\mathbb S \mathbb D (\psi) \,\mathbb D^2 P_n + \mathbb
S^2 (\psi) \,\mathbb D^2 P_n  \big] - \mathbb S \mathbb D (\psi)
\,\mathbb S \mathbb D P_n.
\end{multline*}
We also have
\begin{equation*}
\mathbb D^2 \big[\psi \, P_n  \big] = 2 \mathbb S \mathbb D (\psi) \,\mathbb S \mathbb D P_n
+ \big[ m_1\,\mathbb S \mathbb D (\psi) + \mathbb S^2 (\psi) \big]\,\mathbb D^2 P_n .
\end{equation*}
Therefore,
\begin{multline}\label{eq:2star}
2 \mathbb D^2 \big[\psi \,m_2(t)\,\mathbb D^2 P_n  \big] + \mathbb D^2 \big[ \psi \, P_n  \big] \\
=
4 m_1\,\mathbb D^2 \big[\psi \,\mathbb S \mathbb D P_n \big] + 2
\mathbb S \mathbb D \big[\psi \,\mathbb S \mathbb D P_n \big] -
m_1\,\mathbb S \mathbb D (\psi) \,\mathbb D^2 P_n - \mathbb S^2
(\psi) \,\mathbb D^2 P_n.
\end{multline}
{}From \eqref{eq:1star} and \eqref{eq:2star} it yields
\begin{multline*}
\mathbb D^2 \big[2 \phi \mathbb S \, \mathbb D P_n + 2 \psi \,m_2(t)\,\mathbb D^2 P_n   + \psi \, P_n \big] \\
= 4 m_1 \,\mathbb D^2 \big[ \phi \,\mathbb D^2 P_n \big]  + 2\mathbb S \mathbb D \big[\phi \,\mathbb D^2 P_n \big] + 2\mathbb S \mathbb D (\phi) \,\mathbb D^2 P_n + 2\mathbb D^2(\phi)\,\mathbb S \mathbb D P_n \\
+ 4 m_1\,\mathbb D^2 \big[\psi \,\mathbb S \mathbb D P_n \big] + 2 \mathbb S \mathbb D \big[\psi \,\mathbb S \mathbb D P_n \big]
- m_1\,\mathbb S \mathbb D (\psi) \,\mathbb D^2 P_n - \mathbb S^2 (\psi) \,\mathbb D^2 P_n  \\
= 4 m_1\,\mathbb D^2\big[ \lambda_n P_n\big] + 2\mathbb S \mathbb D \big[\lambda_n P_n \big] + 2\mathbb S \mathbb D (\phi) \,\mathbb D^2 P_n + 2\mathbb D^2(\phi)\,\mathbb S \mathbb D P_n \\
-m_1\,\mathbb S \mathbb D (\psi) \,\mathbb D^2 P_n - \mathbb S^2 (\psi) \,\mathbb D^2 P_n.
\end{multline*}
So
\begin{multline}\label{eq:3star}
 \mathbb D^2 \big[2 \phi \mathbb S \, \mathbb D P_n + 2 \psi
\,m_2(t)\,\mathbb D^2 P_n   + \psi \, P_n \big] \\= 4 m_1\lambda_n
\,\mathbb D^2 P_n + 2\lambda_n \,\mathbb S \mathbb D P_n  +
2\mathbb S \mathbb D (\phi) \,\mathbb D^2 P_n + 2\mathbb
D^2(\phi)\,\mathbb S \mathbb D
P_n \\
- m_1\,\mathbb S \mathbb D (\psi) \,\mathbb D^2 P_n -
\mathbb S^2 (\psi) \,\mathbb D^2 P_n.
\end{multline}
Moreover,
\begin{multline}\label{eq:V)}
\mathbb S \mathbb D \big[\phi \,\mathbb S \mathbb D P_n \big] =\mathbb S \mathbb D (\phi) \,\mathbb S^3 \mathbb D P_n + \mathbb S^2 (\phi)\, \mathbb S \mathbb D \mathbb S \mathbb D P_n + m_2(t) \big[\mathbb D^2 (\phi) \, \mathbb D\mathbb S^2 \mathbb D P_n + \mathbb D \mathbb S (\phi) \mathbb D^2 \mathbb S \mathbb D P_n \big]\\
= \mathbb S \mathbb D (\phi)\,\big[(m_2(t) + 2 m_1 ^2) \,\mathbb S \mathbb D^3 P_n  +  3 m_1 m_2(t)\,\mathbb D^4 P_n + m_1 \mathbb D^2 P_n + \mathbb S \mathbb D P_n \big] + \\
\mathbb S^2 (\phi)\,\big[m_2(t) \mathbb D^4 P_n + 2 m_1\,\mathbb S \mathbb D^3 P_n + \mathbb D^2 P_n \big] \\
+ m_2(t) \,\mathbb D^2 (\phi)\,\big[(m_2(t) + 2 m_1 ^2)\,\mathbb D^4 P_n + 3 m_1\,\mathbb S \mathbb D^3 P_n + \mathbb D^2 P_n \big] \\
+ m_2(t) \,\mathbb D \mathbb S (\phi)\,\big[\mathbb S \mathbb D^3 P_n + 2 m_1 \,\mathbb D^4 P_n \big]\\
= \big[3 m_1 m_2(t)\, \mathbb S \mathbb D (\phi)  + m_2(t) \,\mathbb S^2 (\phi)  + m_2^2(t) \,\mathbb D^2 (\phi)  \\
+ 2 m_1^2m_2(t) \,\mathbb D^2 (\phi) +  2 m_1 m_2(t) \,\mathbb D \mathbb S (\phi)  \big] \,\mathbb D^4 P_n \\
+ \big[m_2(t)\, \mathbb S \mathbb D (\phi) + 2 m_1 ^2\, \mathbb S \mathbb D (\phi) + 2 m_1 \,\mathbb S^2 (\phi)   \\
+  3 m_1 m_2(t) \,\mathbb D^2 (\phi) +  m_2(t) \,\mathbb D \mathbb S (\phi) \big] \mathbb S \mathbb D^3 P_n\\
 + \big[ m_1\,\mathbb S \mathbb D (\phi) + \mathbb S^2 (\phi)  +  m_2(t) \,\mathbb D^2 (\phi) \big]
\,\mathbb D^2 P_n + \mathbb S \mathbb D (\phi)\, \mathbb S \mathbb D P_n\\
= \big[ m_2(t)\,\big( m_1\,\mathbb S \mathbb D (\phi)  + m_2(t)\,\mathbb D^2 (\phi) + \mathbb S^2 (\phi)\big) \\
+ 2 m_1\,\big(m_2(t)\,\mathbb S \mathbb D (\phi)  + m_1m_2(t)\, \mathbb D^2 (\phi) + m_2(t)\,\mathbb D \mathbb S (\phi) \big) \big] \,\mathbb D^4 P_n \\
+ \big[ m_2(t)\,\big( \mathbb S \mathbb D (\phi)  + \mathbb D \mathbb S (\phi) + m_1\,\mathbb D^2 (\phi)\big) \\
+ 2 m_1\,\big(m_2(t)\,\mathbb D^2 (\phi)  + m_1\,\mathbb S \mathbb D (\phi) + \mathbb S^2 (\phi) \big) \big] \,\mathbb S \mathbb D^3 P_n \\
 + \big[ m_1\,\mathbb S \mathbb D (\phi) + \mathbb S^2 (\phi)  +  m_2(t) \,\mathbb D^2 (\phi) \big] \,\mathbb D^2 P_n + \mathbb S \mathbb D (\phi)\, \mathbb S \mathbb D P_n\\
 = m_2(t)\,\mathbb D^2 \big[\phi \, \mathbb D^2 P_n \big] + 2 m_1\,\mathbb S \mathbb D \big[ \phi \,\mathbb D^2 P_n \big] -  m_2(t)\,\mathbb D^2 (\phi) \,\mathbb D^2 P_n  - 2m_1\, \mathbb S \mathbb D (\phi) \,\mathbb D^2  P_n\\
 +  \big[ m_1\,\mathbb S \mathbb D (\phi) + \mathbb S^2 (\phi)  +  m_2(t) \,\mathbb D^2 (\phi) \big] \,\mathbb D^2 P_n  + \mathbb S \mathbb D (\phi)\, \mathbb S
\mathbb D P_n.
\end{multline}
Thus,
\begin{multline}\label{eq:4star}
\mathbb S \mathbb D \big[ \phi \,\mathbb S \mathbb D P_n \big]  \\
=  m_2(t) \,\mathbb D^2 \big[\phi \, \mathbb D^2 P_n \big]
+ 2 m_1\,\mathbb S \mathbb D \big[ \phi \,\mathbb D^2 P_n \big] +
 \big[m_2(t) \,\mathbb D^2 (\phi) + \phi  \big] \,\mathbb D^2 P_n + \mathbb S \mathbb D (\phi)\, \mathbb S \mathbb D P_n.
\end{multline}
Also,
\begin{multline}\label{eq:VI)}
\mathbb S \mathbb D \big[\psi\,m_2(t) \,\mathbb D^2 P_n \big] = \mathbb S \mathbb D (\psi\,m_2(t))\, \mathbb S^2 \mathbb D^2 P_n + \mathbb S^2 (\psi\,m_2(t))\, \mathbb S \mathbb D^3 P_n \\
+ m_2(t)\,\mathbb D^2 (\psi\,m_2(t))\, \mathbb D \mathbb S \mathbb D^2 P_n +m_2(t)\,\mathbb D \mathbb S (\psi\,m_2(t))\, \mathbb D^4 P_n \\
= \big[m_2(t)\,\mathbb S \mathbb D (\psi)  + 2 m_1 ^2\,\mathbb S \mathbb D (\psi) + 2 m_1\,\mathbb S^2 (\psi)  \big] \big[m_2(t)
\mathbb D^4 P_n + m_1\,\mathbb S \mathbb D^3 P_n + \mathbb D^2 P_n \big] \\
+ \big[4 m_1 m_2(t)\,\mathbb S \mathbb D (\psi)  + 2 m_1 ^3\,\mathbb S \mathbb D (\psi) + (2 m_1^2  + m_2(t))\,\mathbb S^2 (\psi)   \big]\,\mathbb S \mathbb D^3 P_n \\
+  4 m_1 m_2(t)\,\mathbb S \mathbb D (\psi) \,\big[ \mathbb S \mathbb D^3 P_n + m_1 \,\mathbb D^4 P_n  \big]
\\ + m_2(t) \big[m_2(t)\,\mathbb S \mathbb D (\psi)  + 6 m_1 ^2\,\mathbb S \mathbb D (\psi) + 2 m_1\,\mathbb S^2 (\psi)   \big] \,\mathbb D^4 P_n \\
= \big[2m_2^2(t)\,\mathbb S \mathbb D (\psi)  + 12 m_1 ^2 m_2(t)\,\mathbb S \mathbb D (\psi) + 4 m_1 m_2(t)\,\mathbb S^2
(\psi)    \big] \,\mathbb D^4 P_n\\
 + \big[9 m_1 m_2(t)\,\mathbb S \mathbb D (\psi)  + 4 m_1 ^3\,\mathbb S \mathbb D (\psi) + 4 m_1^2\,\mathbb S^2 (\psi) + m_2(t)\,\mathbb S^2 (\psi)  \big]\, \mathbb S \mathbb D^3 P_n \\
 + \big[m_2(t)\,\mathbb S \mathbb D (\psi) + 2 m_1^2\,\mathbb S \mathbb D (\psi) + 2 m_1\,\mathbb S^2 (\psi) \big]\, \mathbb D^2 P_n,
\end{multline}
and
\begin{multline*}
\mathbb S \mathbb D \big[\psi\, P_n \big] = \mathbb S \mathbb D
(\psi)\, \mathbb S^2 P_n + \mathbb S^2 (\psi)\, \mathbb S \mathbb
D P_n
+ m_2(t) \,\mathbb D \mathbb S (\psi)\, \mathbb D^2 P_n \\
= 2 m_2(t) \,\mathbb S \mathbb D (\psi)\, \mathbb D^2 P_n +
m_1\,\mathbb S \mathbb D (\psi) \,\mathbb S \mathbb D P_n +
\mathbb S^2 (\psi)\, \mathbb S \mathbb D P_n + \mathbb S \mathbb D
(\psi)\,P_n.
\end{multline*}
As a consequence,
\begin{multline*}
\mathbb S \mathbb D \big[\psi\,m_2(t) \, \mathbb D^2 P_n \big] +
\mathbb S \mathbb D \big[\psi\, P_n \big] = m_2(t) \,\mathbb D^2
\big[ \psi\,\mathbb S \mathbb D P_n \big] + 2 m_1\,\mathbb S
\mathbb D \big[ \psi\,\mathbb S \mathbb D P_n \big]\\ -
m_1\,\mathbb S \mathbb D (\psi) \,\mathbb S \mathbb D P_n + m_2(t)
\,\mathbb S \mathbb D (\psi) \mathbb D^2 P_n + \mathbb S^2
(\psi)\, \mathbb S \mathbb D P_n + \mathbb S \mathbb D
(\psi)\,P_n.
\end{multline*}
Since $ \mathbb S^2(\psi) = m_1\,\mathbb S \mathbb D (\psi) + \psi$, we have
\begin{multline}\label{eq:5star}
2 \mathbb S \mathbb D \big[\psi\,m_2(t) \,\mathbb D^2 P_n \big] +
\mathbb S \mathbb D \big[\psi\, P_n \big] \\= 2 m_2(t) \,\mathbb D^2
\big[ \psi\,\mathbb S \mathbb D P_n \big] + 4 m_1\,\mathbb S
\mathbb D \big[ \psi\,\mathbb S \mathbb D P_n \big] - 2
m_1\,\mathbb S \mathbb D (\psi) \,\mathbb S \mathbb D P_n\\  +
\psi \,\mathbb S \mathbb D P_n + \mathbb S \mathbb D (\psi)\,P_n.
 \end{multline}
{}From \eqref{eq:4star} and  \eqref{eq:5star} it yields
\begin{multline*}
\mathbb S \mathbb D \big[ 2 \phi \,\mathbb S \mathbb D P_n + 2
\psi\,m_2(t) \,\mathbb D^2 P_n + \psi\,P_n \big] = 2m_2(t)
\,\mathbb D^2 \big[\phi \, \mathbb D^2 P_n \big] + 4 m_1\,\mathbb
S \mathbb D \big[ \phi \,\mathbb
D^2 P_n \big] \\
 + 2  \big[m_2(t) \,\mathbb D^2 (\phi) + \phi  \big] \,\mathbb D^2 P_n + 2 \mathbb S \mathbb D (\phi)\, \mathbb S \mathbb D P_n
+ 2 m_2(t) \,\mathbb D^2 \big[ \psi\,\mathbb S \mathbb D P_n
\big]\\
+ 4 m_1\,\mathbb S \mathbb D \big[ \psi\,\mathbb S \mathbb
D P_n \big]
- 2 m_1\,\mathbb S \mathbb D (\psi) \,\mathbb S \mathbb D P_n + \psi \,\mathbb S \mathbb D P_n + \mathbb S \mathbb D (\psi)\,P_n\\
= 2 m_2(t) \,\mathbb D^2 \big[ \lambda_n  P_n \big] + 4 m_1\,\mathbb S \mathbb D \big[ \lambda_n P_n \big]
+ 2  \big[m_2(t) \,\mathbb D^2 (\phi) + \phi  \big] \,\mathbb D^2 P_n + 2 \mathbb S \mathbb D (\phi)\, \mathbb S \mathbb D P_n \\
- 2 m_1\,\mathbb S \mathbb D (\psi) \,\mathbb S \mathbb D P_n + \psi  \,\mathbb S \mathbb D P_n + \mathbb S \mathbb D (\psi)\,P_n.
\end{multline*}
Since $-  m_1\,\mathbb S \mathbb D (\psi) = \psi - \mathbb S^2 (\psi)$, we have
\begin{multline}\label{eq:6star}
\mathbb S \mathbb D \big[ 2 \phi \,\mathbb S \mathbb D P_n + 2
\psi\,m_2(t) \,\mathbb D^2 P_n + \psi\,P_n \big] \\ = 2 m_2(t)\,
\lambda_n  \,\mathbb D^2 P_n + 4 m_1\,\lambda_n\, \mathbb S
\mathbb D P_n  + 2  m_2(t) \,\mathbb
D^2 (\phi) \,\mathbb D^2 P_n \\
+  2 \lambda_n\,P_n + 2 \mathbb S
\mathbb D (\phi)\, \mathbb S \mathbb D P_n - m_1\,\mathbb S
\mathbb D (\psi) \,\mathbb S \mathbb D P_n - \mathbb S^2 (\psi)
\,\mathbb S \mathbb D P_n + \mathbb S \mathbb D (\psi)\,P_n.
 \end{multline}
{}From \eqref{eq:3star} and \eqref{eq:6star} it yields
\begin{multline*}
   \phi\,\mathbb D^2 \big[2 \phi \mathbb S \, \mathbb D P_n + 2 \psi \,m_2(t)\,\mathbb D^2 P_n   + \psi \, P_n \big]
    +  \psi\,\mathbb S \mathbb D \big[ 2 \phi \,\mathbb S \mathbb D P_n + 2 \psi\,m_2(t) \,\mathbb D^2 P_n + \psi\,\,P_n \big] = \\
4 m_1\lambda_n    \phi\,\mathbb D^2 P_n + 2\lambda_n \phi\,\mathbb S \mathbb D P_n  + 2   \phi\,\mathbb S \mathbb D (\phi) \,\mathbb D^2 P_n + 2   \phi\,\mathbb D^2(\phi)\,\mathbb S \mathbb D P_n  -  m_1   \phi\,\mathbb S \mathbb D (\psi) \,\mathbb D^2 P_n  \\
- \phi\,\mathbb S^2 (\psi) \,\mathbb D^2 P_n + 2 m_2(t) \, \lambda_n \psi \,\mathbb D^2 P_n + 4 m_1\lambda_n
\psi \, \mathbb S \mathbb D P_n  + 2  m_2(t) \psi \,\mathbb D^2 (\phi) \,\mathbb D^2 P_n \\
+  2 \lambda_n\psi \,P_n + 2\psi\,\mathbb S \mathbb D (\phi)\, \mathbb S \mathbb D P_n - m_1\psi \,\mathbb
S \mathbb D (\psi) \,\mathbb S \mathbb D P_n -\psi \,\mathbb S^2 (\psi) \,\mathbb S \mathbb D P_n + \psi \,\mathbb S \mathbb D (\psi)\,P_n\\
= 4 m_1\lambda_n \big[\phi\,\mathbb D^2 P_n + \psi \, \mathbb S \mathbb D P_n\big] + 2\lambda_n \phi\,\mathbb S \mathbb D P_n
 + 2 \mathbb S \mathbb D (\phi) \big[\phi\,\mathbb D^2 P_n + \psi\, \mathbb S \mathbb D P_n\big]  + 2   \phi\,\mathbb D^2(\phi)\,\mathbb S \mathbb D P_n  \\
- m_1 \mathbb S \mathbb D (\psi) \big[\phi \,\mathbb D^2 P_n +
\psi\,\mathbb S \mathbb D P_n \big]  -  \mathbb S^2 (\psi) \big[
\phi\,\mathbb D^2 P_n + \psi\,\mathbb S \mathbb D P_n \big]  + 2
m_2(t) \, \lambda_n \psi \,\mathbb D^2 P_n +  2 \lambda_n\psi
\,P_n\\
+ 2  m_2(t) \psi \,\mathbb D^2 (\phi) \,\mathbb D^2 P_n + \psi \,\mathbb S \mathbb D (\psi)\,P_n \\
= \big[ 4 m_1\,\lambda_n ^2 +  2\lambda_n \,\mathbb S \mathbb D (\phi)  - m_1\,\lambda_n \,\mathbb S \mathbb D (\psi) - \lambda_n \, \mathbb S^2 (\psi)
+  2 \lambda_n\,\psi  + \psi \,\mathbb S \mathbb D (\psi) \big] \, P_n \\
+ \big[2 \lambda_n  + 2 \mathbb D^2(\phi) \big] \,\phi \,\mathbb S \mathbb D P_n  + 2 \big[ \lambda_n
+ \mathbb D^2 (\phi)  \big] \, m_2(t) \psi \,\mathbb D^2 P_n.
\end{multline*}
So,
\begin{multline*}
\operatorname L_{n-1} \big[2 \phi \mathbb S \, \mathbb D P_n + 2 \psi \,m_2(t)\,\mathbb D^2 P_n   + \psi \, P_n \big] =\\
\big[ 4 m_1\,\lambda_n ^2 +  2\lambda_n \,\mathbb S \mathbb D (\phi)  - m_1\,\lambda_n \,\mathbb S \mathbb D (\psi) - \lambda_n \, \mathbb S^2 (\psi)+  2 \lambda_n\,\psi  + \psi \,\mathbb S \mathbb D (\psi) \big] \, P_n \\
+ \big[2 \lambda_n  + 2 \mathbb D^2(\phi) \big] \,\phi \,\mathbb S \mathbb D P_n  + 2 \big[ \lambda_n + \mathbb D^2 (\phi)  \big] \,m_2(t) \psi \,\mathbb D^2 P_n \\
-  \lambda_{n-1}\big[2 \phi \mathbb S \, \mathbb D P_n + 2 \psi \,m_2(t)\,\mathbb D^2 P_n   + \psi \, P_n \big] \\
= \big[ 4 m_1\,\lambda_n ^2 +  2\lambda_n \,\mathbb S \mathbb D (\phi)  - m_1\,\lambda_n \,\mathbb S \mathbb D (\psi) - \lambda_n \, \mathbb S^2 (\psi) +  2 \lambda_n\,\psi  + \psi \,\mathbb S \mathbb D (\psi)  - \lambda_{n-1}\,\psi  \big] \, P_n \\
+ \big[2 \lambda_n  + 2 \mathbb D^2(\phi)  -  2 \lambda_{n-1} \big] \,\phi \,\mathbb S \mathbb D P_n
+  \big[ 2\lambda_n + 2\mathbb D^2 (\phi) - 2\lambda_{n-1} \big] \, m_2(t) \psi \,\mathbb D^2 P_n. \\
 \end{multline*}
{}From \eqref{eq:I)} and \eqref{eq:II)} we have
\begin{multline*}
\operatorname L_{n-1} \big( U_n(\mu(t) \big) \\
= \operatorname L_{n-1} \big[ 2 \phi \mathbb S \, \mathbb D P_n + \psi P_n +
(\lambda_n - \lambda_{n+1} )(\mu(t) - \beta_n ) P_n +  \lambda_n P_n + 2 m_2(t)\,\psi \mathbb D^2 P_n\big]  \\
 = \big[ 4 m_1\,\lambda_n ^2 +  2\lambda_n \,\mathbb S \mathbb D (\phi)  - m_1\,\lambda_n \,\mathbb S \mathbb D (\psi)
 - \lambda_n \, \mathbb S^2 (\psi) +  2 \lambda_n\,\psi  + \psi \,\mathbb S \mathbb D (\psi)   - \lambda_{n-1}\,\psi\\
 +  \lambda_n(\lambda_n - \lambda_{n-1})  + (\lambda_n - \lambda_{n+1} )(\lambda_n -
\lambda_{n-1} )(\mu(t) - \beta_n )  + (\lambda_n - \lambda_{n+1} )\psi  \big] \, P_n  \\
+ \big[2 \lambda_n  + 2 \mathbb D^2(\phi)  - 2 \lambda_{n-1}  + 2(\lambda_n - \lambda_{n+1} ) \big] \,\phi
\,\mathbb S \mathbb D P_n\\
+ \big[ 2\lambda_n + 2\mathbb D^2 (\phi) - 2\lambda_{n-1}  + 2(\lambda_n - \lambda_{n+1})  \big] \, m_2(t) \psi \,\mathbb D^2 P_n .
\end{multline*}
Since $\lambda_{n+1} = (n + 1)(a_{0}n + b_{0})$, we have
\[
\lambda_n - \lambda_{n-1} = 2a_{0}n - 2a_{0} + b_{0} , \qquad \lambda_n - \lambda_{n+1} = -2a_{0}n -b_{0},
\]
and
\[
\lambda_n + \mathbb D^2 (\phi) - \lambda_{n-1}  + \lambda_n - \lambda_{n+1}  = 0.
\]
As a consequence,
\begin{multline*}
\operatorname L_{n-1} \big( U_n(\mu(t) \big)
 =  \big[ 4 m_1\,\lambda_n ^2 +  2\lambda_n \,\mathbb S \mathbb D (\phi)  - m_1\,\lambda_n \,\mathbb S \mathbb D (\psi) - \lambda_n \, \mathbb S^2 (\psi)
  +  2 \lambda_n\,\psi  + \psi \,\mathbb S \mathbb D (\psi) \\
    - \lambda_{n-1}\,\psi + \lambda_n(\lambda_n - \lambda_{n-1} )
  + (\lambda_n - \lambda_{n+1} )(\lambda_n -
\lambda_{n-1} )(\mu(t) - \beta_n )  + (\lambda_n
-\lambda_{n+1})\psi  \big] \, P_n,
\end{multline*}
i.e. we have that that $\operatorname L_{n-1} \big( U_n(\mu(t) \big) $ is a polynomial of degree at least $n$. Since $U_n$ is a polynomial of degree $n-1$ and the operator $\operatorname L_{n-1}$  keeps the degree of the polynomials, $\operatorname L_{n-1} \big(U_n(\mu(t) \big)  = 0$.
\end{proof}

\begin{proof}[Proof of Lemma \ref{lemma:43}]
The action of $\operatorname L_{n+1}$ on \eqref{eq:ttrr} gives
\begin{equation*}
0=\operatorname L_{n+1} \big( P_{n+1} \big)=\operatorname L_{n+1}(\mu(t)P_{n}-\beta_{n}P_{n}) - \operatorname L_{n+1}(\gamma_{n}P_{n-1})=U_{n}(\mu(t))-\operatorname L_{n+1}(\gamma_{n}P_{n-1}) .
\end{equation*}
Then, as
\begin{equation*}
\operatorname L_{n+1}\big(P_{n-1} \big) = \big ( \lambda_{n-1} -\lambda_{n+1} \big ) \, P_{n-1},
\end{equation*}
and by \eqref{numerar}, we get \eqref{eq:eqgamma}.
\end{proof}

\begin{proof}[Proof of theorem \ref{theorem:44}]
In order to determine the coefficients $\beta_{n}$ and $\gamma_{n}$ in the three-term recurrence relation \eqref{eq:ttrrpp} in terms of the coefficients $a_{0}$, $a_{1}$, $a_{2}$, $b_{0}$, and $b_{1}$ of the polynomials $\phi$ and $\psi$ given in \eqref{eq:defphi} and \eqref{eq:defpsi} of the divided-difference operator $\operatorname L_{n}$ given in \eqref{eq:bochnernul}, by using \eqref{eq:neweq1} and \eqref{eq:neweq2} we have
\begin{multline*}
\operatorname L_{n}\big(P_{n}(\mu(t))\big) = \operatorname L_{n}\big(\vartheta_{n}(t)\big) + p_{1,n} \operatorname
L_{n}\big(\vartheta_{n-1}(t)\big) + p_{2,n} \operatorname L_{n}\big(\vartheta_{n-2}(t)\big) + \cdots \\
= \big[ k_{0,n}\,\vartheta_{n}(t) + k_{1,n}\,\vartheta_{n-1}(t) + k_{2,n}\,\vartheta_{n-2}(t) \big] \\
+ p_{1,n} \big[k_{0,n-1}\, \vartheta_{n-1}(t) + k_{1,n-1}\,\vartheta_{n-2}(t) + k_{2,n-1}\,\vartheta_{n-3}(t) \big] \\
+ p_{2,n} \big[k_{0,n-2}\,\vartheta_{n-2}(t) + k_{1,n-2}\,\vartheta_{n-3}(t) + k_{2,n-2}\,\vartheta_{n-4}(t) \big] + \cdots\\
= k_{0,n}\,\vartheta_{n}(t) + \big[ k_{1,n} + p_{1,n}k_{0,n-1}\big]\,\vartheta_{n-1}(t) + \big[k_{2,n} + p_{1,n}k_{1,n-1} + p_{2,n}k_{0,n-2} \big]\,\vartheta_{n-2}(t)+ \cdots .
 \end{multline*}
Since
\[
k_{1,n} + p_{1,n}k_{0,n-1} = 0, \qquad k_{2,n} + p_{1,n}k_{1,n-1} + p_{2,n}k_{0,n-2} = 0,
\]
we obtain \eqref{eq:betanfinal} and \eqref{eq:gammanfinal}. Moreover, from the second-order linear divided-difference equation we derive \eqref{eq:p1nfinal}, \eqref{eq:p2nfinal}, and \eqref{eq:lambdanfinal}, which completes the proof. 
\end{proof}


\end{document}